\documentclass{article}

\usepackage{
amsfonts,
amsmath,
amsopn,
amssymb,
amsthm,
authblk,
bbm,
bbold,
comment,
dsfont,
enumitem,
graphicx,
mathrsfs,
mathtools,
mathabx,
soul,
subfig,
verbatim,
xcolor,
xspace,
latexsym,
}

\usepackage[pagewise]{lineno}
\newtheorem{theorem}{Theorem}[section]

\newtheorem{lemma}[theorem]{Lemma}
\newtheorem{remark}[theorem]{Remark}
\newtheorem{proposition}[theorem]{Proposition}

\usepackage[numbers]{natbib}

\newcommand{\R}{\mathbb R}

\newcommand{\G}{\mathcal G}

\newcommand{\NLS}{\rm{NLS}}
\newcommand{\I}{\mathcal I}
\newcommand{\E}{\mathcal E}
\newcommand{\wy}{\widetilde{y}}
\newcommand{\mH}{\mathcal {H}}
\newcommand{\wmH}{\widetilde{\mathcal{H}}}
\newcommand{\whG}{\widehat {\G}}
\newcommand{\wi}{\widetilde{\i}}
\newcommand{\mK}{\mathcal {K}}

\begin{document}

\title{Ground States for the NLS on  non-compact graphs with an attractive potential}

\author[1]{Riccardo Adami \thanks{\tt riccardo.adami@polito.it} }
\author[1]{Ivan Gallo \thanks{\tt ivan.gallo@polito.it}  }

\author[2,3]{David Spitzkopf \thanks{\tt  spitzkopf98@gmail.com}  }

\affil[1]{Dipartimento di Scienze Matematiche “G.L. Lagrange”, Politecnico di Torino
Corso Duca degli Abruzzi, 24, 

10129 Torino, Italy}
\affil[2]{Nuclear Physics Institute, Czech Academy of Sciences, Hlavní 130, 25068 Řež near Prague, Czech Republic }
\affil[3]{Faculty of Mathematics and Physics, Charles University, V Hole\v{s}ovi\v{c}k\'ach 2, 18000 Prague \\ Czech Republic}

\date{ 23 April 2025}

\maketitle

\begin{abstract}
    We consider the subcritical nonlinear Schrödinger equation on non-compact quantum graphs with an attractive potential supported in the compact core, and investigate the existence and the nonexistence of Ground States, defined
    as minimizers of the energy at fixed $L^2$-norm, or mass. We finally reach
    the following picture: for small and large mass there are Ground States. Moreover, according
    to the metric features of the compact core of the graph and to the strength of the 
    potential, there may be an interval of intermediate masses for which there are no Ground States. The study was inspired by the research on quantum waveguides, in which the curvature of a thin tube induces an effective attractive potential.

\end{abstract}

\section{Introduction}

The nonlinear Schrödinger (NLS) equation  is currently used to model the dynamics of several physical systems. 
In particular, it provides an effective description of the evolution of the wave function of a Bose-Einstein condensate (BEC), namely an ultracold system made of many identical bosons in a peculiar phase \cite{Gross1961,pitaevskistringari,
Pitaevskii1961}.
When extended to graphs, the NLS equation can describe evolutionary phenomena in networks, with applications to various contexts in science and engineering  \cite{amico2021roadmap, PhysRevB.84.155304, gnutzmann2011stationary, noja2014,peccianti2008escaping}. In particular,
it describes the behaviour of particles traveling through branched structures of quantum wires in the presence of
nonlinearities.

Here we investigate the existence of ground states for the energy functional
\begin{equation}
    E(u, \mathcal{G}) = \frac{1}{2} \int_{\mathcal{G}} |u'(x)|^2 dx - \frac{1}{p}\int_{\mathcal{G}} |u(x)|^p dx - \frac 1 2 \int_{\mathcal{K}} w(x) |u(x)|^2 dx, \qquad
2 <p<6,    \label{nlsc}
\end{equation}
on a metric graph $\mathcal{G}$ with at least an unbounded edge, i.e. a halfline, under the {mass constraint}
\begin{equation}
    \|u\|^2_{L^2(\mathcal{G})} = \mu ,
    \nonumber
\end{equation}
which means that we seek solutions in the mass constrained space 
\begin{equation}
    H^1_{\mu}(\mathcal{G}) := \{u \in H^1(\mathcal{G}), \, \|u\|^2_{L^2(\mathcal{G})}  = \mu \}.
    \label{h1}
\end{equation}

In Eq \eqref{nlsc} the symbol $\mathcal K$ denotes a compact subset of $\G$ that supports
the action of the continuous potential $w \geq 0$. Notice that with such a choice on the sign of $w$, 
the contribution to Eq \eqref{nlsc}
of the potential term is nonpositive and, since it is supported on a compact set, it goes to zero at infinity and so, roughly speaking,  can be considered as attractive.  For simplicity, we identify $\mathcal K$ with the compact core of $\G$, namely the subgraph made of all bounded edges of $\G$. The meaning of the integrals and of the norms in Eqs \eqref{nlsc} and \eqref{h1} are made precise
in Section \ref{sec:preliminaries}.

The energy functional \eqref{nlsc} is obtained from
the standard NLS energy 
\begin{equation}
     E_{\NLS}(u, \mathcal{G}) = \frac{1}{2} \int_\G
     |u'|^2 \, dx  - \frac{1}{p} \int_\G
     |u|^p \, dx
     \label{functional}
\end{equation}
by adding a term that describes the effect of an external potential $w$. Throughout the paper the potential is chosen as nonnegative, continuous, and, as already mentioned, supported on the compact core $\mathcal{K}$ of the graph $\mathcal{G}$.
%attractive, bounded, and, as already mentioned, supported on a compact subset $\mathcal{K}$ of the graph $\mathcal G$.
The limitation on $p$ means that we 
restrict our analysis to the $L^2$-subcritical case. It is well-known that for $p > 6$ the
constrained functional $E (\cdot, \G)$ is not lower bounded, while the case $p=6$ is more delicate and deserves to be treated separately \cite{cacciapuoti2017}.
Several investigations on the existence of the ground states of (\ref{functional}) on graphs were conducted in recent years, starting from the analysis on
star graphs \cite{acfn10,acfn12} and later extended to general graphs \cite{Ast15,Ast16,ast17, adst19}. Moreover, research has focused on the search for stationary states \cite{SABIROV_2013,Noja_2015} and other special solutions \cite{Sobirov_2010, Kairzhan_2022, Pierotti_Soave_Verzini_2021, dovetta2024normalized}. 
The starting point of the analysis of the NLS on metric graphs is the existence of ground states at every mass for the functional $E_{\NLS}(\cdot , \mathbb{R})$, given by the soliton
\begin{equation}
    \phi_{\mu}(x) = \mu^{\alpha} C_p \, {\rm sech}^{\frac{\alpha}{\beta}}(c_p \mu^{\beta} x)
    \label{soliton} ,
\end{equation}
where $c_p$ and $C_p$ are positive constants and the powers $\alpha$ and $\beta$ are
\begin{equation}
    \label{alfabeta}
    \alpha = \frac{2}{6-p}, \qquad
    \beta = \frac{p-2}{6-p}.
\end{equation}

Up to translations and multiplication by a constant phase,
the function $\phi_\mu$ is the unique
Ground State at mass $\mu$ of $E_{\NLS}$ on the line
\cite{shabat1972exact,berestycki1983nonlinear}.

Let us  recall 
an existence criterion 
singled out in Cor. 3.4 of  \cite{Ast15} for a noncompact graph: \\

\smallskip

\noindent
\textbf{Existence criterion.} \textit{Fix $2 < p < 6$ and let $\mathcal{G}$ be a graph with $n$ infinite edges, $n \geq 1$. If there exists $v \in H^1_{\mu}(\mathcal{G})$ such that }
\begin{equation}
  E_{\NLS}(v, \mathcal{G})  \leq E_{\NLS}(\phi_{\mu}, \mathbb{R}),
  \label{criterion1}
\end{equation}
\textit{then there exists a Ground State for $E_{\NLS}$ on $\mathcal{G}$ at mass $\mu$.} \\

\smallskip \noindent

The criterion provides a conceptually simple method to prove the existence of a Ground State for graphs with an infinite edge: it is sufficient to exhibit a function whose
energy lies below the energy of the soliton. However, in practice the use of such a criterion can be quite cumbersome as the construction of such a function may be non-trivial.  

Furthermore, the converse of the existence criterion holds too, i.e. if a Ground State $v$ at mass $\mu$ exists, then it satisfies Eq \eqref{criterion1}.
The reason is that the soliton Eq \eqref{soliton} can be approximated arbitrarily well by functions supported on the infinite edge, so  the infimum of the energy at a given mass, and therefore energy of a possible ground state, cannot exceed the quantity $E_{\NLS} (\R)$ (Theorem 2.2 in \cite{Ast15}). This remains true for the energy functional \eqref{nlsc}, where the presence of an attractive potential further lowers the infimum of the constrained energy.

A preliminary result of the present paper consists in the extension
of the existence criterion to 
the case with the potential $w$ (Lemma \ref{genex}). We exploit then such generalization
to prove two results of existence and one of nonexistence of ground states.

The first is Theorem \ref{large mass}, that
establishes
that if one concentrates a soliton in a region
where the potential gives a negative
contribution, then the existence criterion is fulfilled and a Ground State exists. Of course, this can be done if the mass is fixed as large enough, in order to allow to confine the soliton in such a region.

On the other side of the range of the mass, the content of Theorem \ref{small mass}, namely the existence of ground states for small mass, is not new, as it
was already proved in \cite{Noja_2017} for every attractive potential.
Such ground states were proven to arise
as a nonlinear bifurcation  from  linear ground states. Here we give a different proof based on the extension of the
existence criterion
to the functional $E(u, \mathcal{G})$ (Lemma \ref{genex}). 

The last achievement, proven in Theorem \ref{th:nonex}, is a nonexistence result that holds for a class
of graphs and some interval of masses, under
the hypothesis of a weak potential. To our
knowledge this is the first nonexistence
theorem for an attractive potential on graphs,
except for the case of a Dirac's delta potential,
and generalizes a nonexistence result in \cite{Ast15} that holds in the absence of potentials.
Owing to Theorems \ref{large mass} and \ref{small mass} the picture we get is the following:
in the presence of a negative potential there are two mass thresholds $\mu_\star \leq \mu^\star$. Existence is guaranteed
below $\mu_\star$ and above $\mu^\star$. Between the two, there
may be an interval of masses in which ground states do not exist. The construction of explicit examples in which the nonexistence occurs is not straightforward, and we exhibit one such example in Section 5,
inspired by the analogous result 
proven in \cite{Ast16}. A more detailed investigation on nonexistence 
is in order, requiring a more thorough analytical and numerical effort.

Besides the case of an external potential,
our analysis applies to the case of a potential induced by the curvature of the graph \cite{DEx95}. 
More specifically, we recall that
for a waveguide modeled as a curved tube $\Omega_{\Gamma}$ around a curve $\Gamma$, one can express the Laplacian
with Dirichlet boundary conditions by using a transformation from the ordinary cartesian coordinates to a system centred in a straightened tube $\Omega_0$. This 
procedure results in  an effective operator
with a potential that depends on the signed curvature $\gamma (x) $, where $x$ is the longitudinal coordinate that coincides with the arclength,
i.e.,
\begin{equation}
    V(x,\xi) := -\frac{\gamma^2 (x)}{4(1+\xi \gamma (x))^2} + \frac{\xi {\gamma'' (x)}}{2(1+\xi\gamma (x))^3} + \frac{5}{4} \frac{\xi^2 {\gamma' (x)}}{(1+\xi \gamma (x))^4},
    \nonumber
\end{equation}
where  $\xi$ is the transverse coordinate.
In the thin-waveguide limit $\xi \to 0$ one has $V(x) \approx - \frac{\gamma^2 (x)}{4}$, and the transversal
contribution remains present in the effective potential for finite $\xi$. When we try to rigorously apply this limit to branched networks, the main source of issues is the behaviour of eigenfunctions  
%in and locally 
around the vertices (see Sec. 8.3--8.5 in \cite{exner2015quantum}).  Here we circumvent this complication by avoiding vertices, and get inspired by this limit to study the effects of adding a purely attractive potential to the edges of a generic graph $\mathcal{G}$. 

The paper is organized as follows:
in Section 2 we recall some well-known estimates and give some preliminary results.
The goal of Section 3 is to prove Lemma \ref{genex},
namely the existence criterion for the functional $E(u, \mathcal{G})$.
In Section 4 we exploit Lemma \ref{genex} and prove the existence of ground states for large (Theorem \ref{large mass})  and small (Theorem \ref{small mass}) masses.
In Section 5 we prove a nonexistence theorem (Theorem \ref{th:nonex}) that identifiess a class of graphs for which, for some values of the mass, ground states do not exist despite the presence of the potential.

\section{General framework}
\label{sec:preliminaries}

A metric graph $\mathcal G$ embedded in an Euclidean space
of dimension $d$ is a metric space defined by the pair $(\mathcal V, \E)$, where $\mathcal V \subset \R^d$ is a set of points called {vertices}, and $\E \subset \mathcal V \times \mathcal V$ is a set whose elements are called edges, that are interpreted as connections between couples of vertices. 
The metric  structure is defined by establishing that every edge $e \in \E$ is homeomorphic to a real interval $I_e = [0, \ell_e]$. If $\ell_e = \infty$, then the 
edge $e$ is a halfline and $\mathcal G$ is  
noncompact. We restrict to the case of 
a finite number of edges and vertices (that excludes in particular the case of periodic 
graphs) so in this context noncompact means
that at least one edge is a halfline. We denote by $\mathcal H_i$
the $i.$th halfline of $\G$. We shall not distinguish in notation between a point in $\G$ and its coordinate $x$ in the corresponding interval $I_e.$ Moreover, we assume that $\G$ is connected.

A function $u : \mathcal G \to \mathbb C$ is
a family of functions $u_e : I_e \to \mathbb C$, one for every edge, namely
$u \equiv \{ u_e\}_{e \in \E}.$
Differential operators are naturally defined edge by edge, namely
$u' = \{ u_e'\}_{e \in \E},$
and integrals on $\G$ are defined by
$$ \int_\G u (x) \, dx \ = \ \sum_{e \in \E} \int_0^{\ell_e} u (x_e) \, dx_e.$$

Furthermore, $L^p$-spaces are defined in the
natural way, i.e.,
$$L^p (\G) = \{ u : \mathcal G \to \mathbb C,
\ \| u_e \|_{L^p (I_e)} < \infty \ \forall e \in \E \ \}, \quad \| u \|_{L^p (\mathcal G)}^p = \sum_{e \in \E} \| u_e \|_{L^p (I_e)}^p,$$
and analogously for the Sobolev spaces, in particular
$$H^1 (\G) = \{ u : \mathcal G \to \mathbb C,
\ \| u_e \|_{H^1 (I_e)} < \infty \ \forall e \in \E \}, \quad \| u \|_{H^1 (\mathcal G)}^2 = \sum_{e \in \E} \| u_e \|_{H^1 (I_e)}^2 .$$

The space $H^1_\mu (\G)$, already mentioned in Sec. 1, in which we seek the minimizers of the functional $E(\cdot, \G),$ is then
defined as
$$H^1_\mu (\G) = \{ u, \ u \in H^1_\mu (\G), \ \| u \|_{L^2 (\G)}^2 = \mu \}.$$

For a more comprehensive introduction to variational calculus on metric graphs, see e.g., \cite{Ast15}.

Typically one defines the compact core of a graph as the complement of the halflines. In
principle, this notion does not coincide with that of the support of a given compactly supported potential. However, one can consider
a larger compact set by taking the union $\mathcal K'$ of the compact core of the graph with the support of the potential. The complement of $\mathcal K'$ contains then $n$ halflines $\mathcal H'_i$, which in general are subsets of the original $n$ halflines $\mathcal H_i$. 
The complement of the union of the halflines $\mathcal H_i'$ contains both the compact core of the graph $\G$ and the support of $w$, and can
be used in the results of the paper in the place of $\mathcal K$. For the sake of simplicity, in the following
we refer to $\mathcal K$, the compact core of the graph, as the support of the potential. Of course, all results remain valid if the compact support of $w$ does not coincide with $\mathcal K.$

In summary, the hypotheses on $w$ are the following: 
\begin{center}
    $w$ is continuos, nonnegative, and supported on $\mathcal{K}$, the compact core of $\G$.
\end{center} 

\noindent

For notational purposes we introduce the symbols 
\begin{equation*}
T(u) = \int_{\mathcal{G}} |u'|^2 \, dx, \quad \,
V(u) = \int_{\mathcal{G}} |u|^p \, dx, \quad \, 
W(u) = \int_{\mathcal{K}} w(x) |u|^2  \,dx,
\end{equation*}
so that
$$E (u, \G) = \frac 1 2 T (u) - \frac 1 p V (u) - \frac 1 2 W (u).$$

Furthermore, we denote
\begin{align*}
    \I_\G (\mu) & : = \inf_{u \in H^1_\mu (\G)} E (u, \G) \\
    \I_{\NLS,\G} (\mu) &: = \inf_{u \in H^1_\mu (\G)} 
    E_{\NLS} (u, \G).
\end{align*}
\begin{remark}\label{gsbound}
{\em
As mentioned in Sec. 1, the unique positive solution to the problem of the mass-constrained minimization of the functional (\ref{functional}) in the case $\mathcal{G} = \mathbb{R}$ is given, up to translations and multiplications by a constant phase, by soliton $\phi_\mu$ defined in Eq \eqref{soliton}, so
   
\begin{equation}
     \I_{\rm{NLS},{\mathbb{R}}}(\mu) =  E_{\rm{NLS}} (\phi_\mu, \mathbb{R})  = -\theta_p \mu ^{2 \beta+1} < 0
    \label{solitonenergy},
    \nonumber
\end{equation}
where $\theta_p := - E_{\NLS}(\phi_1, \mathbb{R}) > 0$. }
\end{remark}

\begin{remark} \label{gs} { \em
As mentioned in Sec. 1 and proved in 
\cite{Ast15}, if $\mathcal{G}$ is a noncompact graph then \begin{equation}
   \I_{\NLS, \G} (\mu) \leq - \theta_p \mu^{2\beta+1} .
    \nonumber
\end{equation}

Since we are considering a nonnegative $w$, we get
\begin{equation}
     %\inf_{u \in H^1_{\mu}(\mathcal{G})}E(\mu, \mathcal{G}) =
     \I_{\G}(\mu) \leq\I_{\NLS, \G}(\mu) \leq - \theta_p \mu^{2 \beta +1}.
     \label{upperb}
\end{equation}
 }

\end{remark}

\noindent

Occasionally we use the shorthand notation
$$\| w \|_\infty : = \max_{x \in \mathcal K} | w (x) |.$$

\subsection{General Properties}

In this section we recall some well-known properties that are widely used in the context of the nonlinear Schrödinger equation on noncompact graphs.

\begin{proposition}[Gagliardo-Nirenberg inequalities]
    There exists $M_p > 0$  such that
\begin{align}
    \|u\|^p_{L^p(\mathcal{G})} \leq M_p \|u\|_{L^2(\mathcal{G})}^{\frac{p}{2}+1} \|u'\|^{\frac{p}{2}-1}_{L^2(\mathcal{G})} 
       \label{gn1} 
       .
       \end{align}
       \noindent
       Moreover
       \begin{align}
    \|u\|^2_{L^{\infty}(\mathcal{G})} \leq 2 \|u\|_{L^{2}(\mathcal{G})} \|u'\|_{L^{2}(\mathcal{G})} 
    \label{gn2}
\end{align}
for every $u \in H^1({\mathcal{G}})$ and every metric graph $\mathcal{G}$ with $n$ infinite edges, $n \geq 1$.
\end{proposition}

 We refer for proofs to Proposition 4.1 in \cite{TENTARELLI} and Proposition 2.1 in \cite{Ast16}.

\begin{proposition} \label{holder}
    Let $\mathcal{K}$ be a compact metric graph with total length $|\mathcal{K}|$, $1 \leq r \leq p \leq \infty$ and $s = 1/r - 1/p$. Then for every $u \in H^1(\mathcal{K})$  
\begin{equation}
    \|u\|_{L^r (\mathcal{K})} \leq \|u\|_{L^p (\mathcal{K})}
    |\mathcal K|^s.
    \nonumber
\end{equation}
\end{proposition}

\begin{proof}
    This is a direct consequence of the Hölder inequality  %(\cite{Holder1889}). 
\begin{equation}
    \int_{\mathcal{K}} | f g| dx \leq (\int_{\mathcal{K}} |f|^p d x)^{\frac{1}{p}} (\int_{\mathcal{K}} |g|^q dx)^{\frac{1}{q}}
    \nonumber
\end{equation}
where $p, q > 0$ and $1/p + 1/q = 1$, if we consider $f = |u|^r$  and $g \equiv 1$ on ${\mathcal K}$.
\end{proof} 
\noindent
\indent Since we need lower and upper bounds for the terms appearing in the energy (\ref{functional}), we prove the following result.
\begin{proposition}
\label{apriori}
Let $\G$ be a metric graph with $n$ infinite edges, $n \geq 1$.
For all $u \in H^1_{\mu}(\mathcal{G})$ such that
\begin{equation}
    E(u, \mathcal{G}) \ \leq \ \frac{1}{2} \, \mathcal I_{\G} (\mu) \ < \ 0,
    \label{half inf}
\end{equation}
 the following estimates hold:
\begin{align}
        \min  \{ C_1 \mu^{2 \beta +1}, C_2 g^{-\frac{2p}{p-2}} \mu^{3 (2 \beta +1)} \} \leq & \ T (u) \leq  C_3 \mu^{2 \beta + 1 } + C_4 \|w\|_{\infty} \mu 
        \label{Tinf}\\
        \min \{ C_1 \mu^{2 \beta +1}, C_2 g^{-\frac{p}{2}} \mu^{(2 \beta +1)\frac{p}{2}} \} \leq &  \ V (u)\leq C_3 \mu^{2 \beta +1} + C_4 \|w\|_{\infty}^{\frac{p-2}{4}} \mu^{\frac{p}{2}}  \label{Vinf} \\
        \min  \{ C_1 \mu^{\beta +1}, C_2 g^{-\frac{p}{p-2}} \mu^{3 \beta +2} \} \leq & \ \|u\|^2_{L^{\infty}(\mathcal{G})} \leq  C_3 \mu^{\beta +1} + C_4 \|w\|_{\infty}^{\frac 1 2} \mu
        \label{uinf}
 \end{align}
 for some constants $C_1, C_2, C_3,C_4 > 0$, where  $g := \|w\|_{\infty} |\mathcal{K}|^{\frac{p-2}{p}}$.
\end{proposition}
\begin{proof}
Let $u \in H^1_{\mu}(\mathcal{G})$ satisfy hypothesis \eqref{half inf}.
If 
\begin{equation}
\label{smallW}
    W (u) <  \frac{1}{2} \theta_p \mu^{2 \beta + 1}
    ,\end{equation} 
\noindent
then by \eqref{half inf}, \eqref{smallW}, and 
\eqref{upperb}

\begin{equation} \begin{split}
     \frac{1}{2} T (u) - \frac{1}{p} V (u) \ =  & \
     E (u, \mathcal G) + \frac 1 2 W (u) \ < \
     \frac 1 2 \mathcal I_{\mathcal{G}} (\mu) +\frac{1}{4} \theta_p \mu^{2 \beta + 1} \\ \leq \ & - 
     \frac{1}{2} \theta_p \mu^{2 \beta + 1}
     +  \frac{1}{4} \theta_p \mu^{2 \beta + 1} \ 
     = \ - \frac{1}{4} \theta_p \mu^{2 \beta + 1} 
     . \label{t-v}
    % \nonumber
     \end{split}
\end{equation}

So the estimates of Lemma 2.6 in \cite{Ast16} hold, namely 
\begin{equation} \begin{split} \label{largemu}
     C_1 \mu^{2 \beta +1} \leq &  T (u) \leq C_3 \mu^{2 \beta + 1 } \\
    C_1 \mu^{2 \beta +1} \leq & V  (u) \leq C_3 \mu^{2 \beta +1}  \\
    C_1 \mu^{\beta +1} \leq  & \|u\|^2_{L^{\infty}(\mathcal{G})} \leq  C_3 \mu^{\beta +1}.
    %    \label{uinf1}
 \end{split} \end{equation}

On the other hand, if 
\begin{equation}
    W (u) \geq \frac{1}{2} \theta_p \mu^{2 \beta + 1},
    \label{Wcase}
\end{equation} 
then we use the inequality \eqref{gn1} and
obtain
\begin{equation}
    V (u) \leq M_p \mu^{\frac{p+2}{4}}T (u)^{\frac{p-2}{4}} 
   \label{gnV}
\end{equation}
that, combined with (\ref{half inf}) yields 
\begin{equation}
    T  (u) - \frac{2 M_p}{p} \mu^{\frac{p+2}{4}} T (u)^{\frac{p-2}{4}} <  W (u) \leq  \|w\|_{\infty} \mu.
    \nonumber
\end{equation}

Since $\frac{p-2}{4} < 1$, {using Young's inequality} there exists $C_4 > 0$
such that
\begin{equation}
    T (u) \leq C_4 \|w\|_{\infty} \mu + C_4 \mu^{2 \beta +1},
   \nonumber
\end{equation}
thus by \eqref{gn1}
\begin{equation}
    V (u) \leq C_4 \|w\|_{\infty}^{\frac{p-2}{4}} \mu^{\frac{p}{2}} + C_4 \mu^{2 \beta +1},
    \nonumber
\end{equation}
and by \eqref{gn2}
\begin{equation}
    \|u\|^2_{L^{\infty}(\mathcal{G})} \leq C_4 \|w\|^{\frac{1}{2}}_{\infty} \mu + C_4 \mu^{\beta +1}.
    \nonumber
\end{equation}
\indent To obtain the lower bound of $V (u)$ we apply Proposition \ref{holder} with $r = 2$. Denoting $g = \|w\|_\infty |\mathcal{K}|^{\frac{p-2}{p}}$ it holds that 
\begin{equation}
    W (u) \leq g V (u)^{\frac{2}{p}} 
    \nonumber
\end{equation}
and from (\ref{Wcase}) one gets
\begin{equation}
    C_2 g^{-\frac{p}{2}} \mu^{(2 \beta +1)\frac{p}{2}} \leq V (u).
    \nonumber
\end{equation}

Using (\ref{gnV}) yields
\begin{equation}
    C_2 g^{-\frac{2p}{p-2}} \mu^{3(2\beta+1)} \leq T (u).
    \nonumber
\end{equation}

Finally, using the fact that $V (u) \leq \mu \|u\|^{p-2}_{\infty}$ we obtain
\begin{equation}
C_2 g^{-\frac{p}{p-2}} \mu^{3 \beta +2}  \leq \|u\|^2_{\infty} .
\nonumber
\end{equation}

Summing up, in the case Eq \eqref{Wcase} we obtain
the inequalities
\begin{equation} \begin{split} \label{smallmu}
 C_2 g^{-\frac{2p}{p-2}} \mu^{3 (2 \beta +1)}  \leq & \ T (u) \leq  C_4 \mu^{2 \beta + 1 } + C_4 \|w\|_{\infty} \mu \\
  C_2 g^{-\frac{p}{2}} \mu^{(2 \beta +1)\frac{p}{2}}  \leq &  \ V (u)\leq C_4 \mu^{2 \beta +1} + C_4 \|w\|_{\infty}^{\frac{p-2}{4}} \mu^{\frac{p}{2}}   \\
         C_2 g^{-\frac{p}{p-2}} \mu^{3 \beta +2}  \leq & \ \|u\|^2_{L^{\infty}(\mathcal{G})} \leq  C_4\mu^{\beta +1} + C_4 \|w\|_{\infty}^{\frac 1 2} \mu.
        \end{split} \end{equation}

By Eqs \eqref{largemu} and \eqref{smallmu} the proof is complete.

\end{proof}

\begin{remark}{\em If $u$ is a Ground State at mass $\mu$, then it satisfies Eq \eqref{half inf}, so
the estimates in Proposition \ref{apriori} hold. Moreover,  from Eq \eqref{smallW} one sees that the effect
of the potential 
does not appear in the estimates if $\mu^{2 \beta} > 2 \frac{\| w \|_{\infty}}{\theta_p}. $ 
Then, the presence of the potential 
is effective in the small mass regime.}

\end{remark}

\begin{remark}
{\em We shall use Proposition \ref{apriori} in the proof of Theorem \ref{th:nonex}. We
need, however, the inequalities Eq \eqref{largemu} only, since the hypotheses of the theorem imply
the assumption Eq \eqref{smallW}. Nonetheless, we chose to prove the more general inequalities Eqs \eqref{Tinf}\text{–}\eqref{uinf}   %Eq \eqref{Vinf}, and 
for the sake of completeness.
}
\end{remark}

In the proof of Theorem \ref{th:nonex} we shall need the notion of monotone rearrangement of a function on a metric graph. Such a notion was first introduced in \cite{friedlander2005extremal}, where the monotone rearrangement is defined for functions on metric graphs by mimicking the analogous classical notion for functions on an interval (see e.g., \cite{Kawohl}). Moreover, in \cite{friedlander2005extremal} the P\'olya-Szeg\H{o} inequality was proved too, stating that the monotone rearrangement does not increase the kinetic energy.  In \cite{Ast15} the theory was extended to the symmetric rearrangement. Here we shall need the monotone rearrangement only,
so we can directly quote the part of Proposition 3.1 in \cite{Ast15} which is relevant for the proof of Theorem \ref{th:nonex}.

\begin{proposition}[Monotone rearrangement] \label{rearrangements}
  
    Let $\G$ be a connected metric graph and let $u \in H^1(\mathcal{G})$ be nonnegative.  Denote by $u^*$ the monotone
    rearrangement of $u$ as defined in \cite{friedlander2005extremal,Ast15}  on the interval $I^* = [0, | \G|),$ where $| \G|$ denotes the total length of $\G$, i.e. the sum, possibly infinite, of the lengths of all edges of $\G.$
Then, $u^* \in H^1 (I^*)$ and
\begin{eqnarray}
    \| u^* \|_{L^r (I^*)} & = &   \| u \|_{L^r (\G)}, \quad r \in [1, +\infty]
    \label{equimeas}\\
     \| (u^*)' \|_{L^2 (I^*)} & \leq &   \| u' \|_{ L^2(\G)}
     \label{polyaszego}. 
\end{eqnarray}
\end{proposition}
Notice that for a general complex-valued function
$u \in H^1(\mathcal{G})$, 
one has  $\| |u| \|_{L^r(\G)} = \| u \|_{L^r (\G)}$ and
$\| |u|' \|_{L^2(\G)} \leq \| u' \|_{L^2 (\G)}$.
Thus, setting $u^* = | u |^*$ both identity \eqref{equimeas} and Pólya-Szeg\H{o} inequality \eqref{polyaszego}
hold.

\begin{remark}{\em (Potential induced by 
curvature)
    From Remark 2.3 in \cite{Ast16} one has that if $u \in H^1(\mathcal{G})$, then the quantities 
\begin{equation}
    \mu^{-2 \beta -1} \|u'\|^2_{L^2(\mathcal{G})}, \hspace{0.5cm} \mu^{-2 \beta -1} \|u\|^p_{L^p(\mathcal{G})}, \hspace{0.5cm} \mu^{-\beta -1} \|u\|^2_{L^{\infty}(\mathcal{G})}
\nonumber
\end{equation}
are invariant under the following rescaling of $\mathcal{G}$ and $u$:
\begin{equation}
    \mathcal{G} \mapsto t^{- \beta} \mathcal{G}, \hspace{0.5cm} u(\cdot) \mapsto t^{\alpha} u (t^{\beta} \cdot).
    \label{scaling}
    \nonumber
\end{equation}

Notice that $u$ is rescaled with the mass as solitons do.
The potential term shows the same invariance if 
one imposes the scaling 
\begin{equation}
    w(\cdot) \rightarrow t^{2 \beta} w(\cdot),
    \label{scalemass}
\end{equation}
i.e., if it scales as the inverse of the square of a length. 
In fact 
\begin{align*}
    \int_{\mathcal{K}} w(x) |u(x)|^2 dx \mapsto &
    \int_{t^{-\beta} \mathcal{K}} t^{2 \beta} w(x) |t^{\alpha}u(t^{\beta} x)|^2 t^{- \beta} d(t^{\beta}x) \\
   = & \ t^{\beta + 2 \alpha} \int_{ \mathcal{K}}  w(x) |u( x)|^2  dx  
\end{align*}
and since 
%\begin{equation}   
 $\beta + 2 \alpha = 2 \beta + 1 
$ from (\ref{alfabeta}), one recovers the same scaling law and 
thus the quantity $\mu^{-2 \beta -1} E(u,\mathcal{G})$ is invariant while the mass  is mapped from $\mu$ to $t \mu$. 

Of course, for a generic potential the scaling assumption \eqref{scalemass} is meaningless. Yet, there is one special case in which it is the correct scaling law,
i.e., the already mentioned case of the potential $- \gamma^2 (x)/4$ induced by the 
presence of a curvature in the edges. Indeed the
curvature $\gamma$ scales as the inverse of a length, namely $\gamma \mapsto t^{\beta} \gamma$, which is consistent with the fact that, by definition, the curvature is the inverse of the radius of the osculating circle.} 
\end{remark}

\section{Existence criterion for $E(\cdot, \G)$}

Here we extend the existence criterion given
in Sec. 1 to the functional (\ref{functional}) by proving the following
\begin{lemma}[Existence criterion for $E (\cdot, \G)$] \label{genex}
Fix $2 < p < 6$ and let $\mathcal{G}$ be a metric graph with $n$ infinite edges, $n \geq 1$. If there exists $v \in H^1_{\mu}(\mathcal{G})$ such that 
\begin{equation}
  E(v, \mathcal{G})  \leq E_{\NLS}(\phi_{\mu}, \mathbb{R}) ,
  \label{criterion}
\end{equation}
then $E (\cdot, \mathcal{G})$ admits a Ground State at mass $\mu$.
\end{lemma}

In order to prove Lemma \ref{genex} we follow the line
of Sec. 3 in \cite{Ast16}.
Thus, as a first step we show that the 
strict concavity  of  $\I_\G$ as a function of $\mu$ is preserved in the presence of the potential term.

\begin{proposition}
\label{subadditive}
 $\mathcal{I}_{\mathcal{G}}$ 
is strictly concave and strictly subadditive  as a function of the mass $\mu$.
\end{proposition}
\begin{proof}

{Consider a sequence $\{v_n \} \subset H^1_\mu (\G)$ that
minimizes $E( \cdot, \G)$ at mass $\mu$. Since $\mathcal I_\G (\mu) < 0,$ eventually $E (v_n, \G) \leq \frac{1}{2} \I_\G (\mu)$. Let us restrict to the elements of the sequence that satisfy such inequality.

By Proposition \ref{apriori}, one has $V (v_n) \geq \min (C_1 \mu^{2 \beta + 1}, C_2 g^{-\frac p 2} \mu^{(2 \beta + 1) \frac p 2})$. 
Let us now introduce the sequence $\{ u_n \} \subset H^1_1 (\G)$, defined as $u_n = \frac{v_n}{\sqrt \mu}.$
Since $V (u_n) = \mu^{-\frac p 2} V (v_n)$, one has $$\mu^{\frac p 2} V (u_n) = V (v_n)
 \geq \min (C_1 \mu^{2 \beta + 1}, C_2 g^{-\frac p 2} \mu^{(2 \beta + 1) \frac p 2}),
$$
that implies that every function $u_n$ belongs to the set 
\begin{equation}
    U:= \left\{ u \in H^1(\mathcal{G}), \int_{\mathcal{G}} |u|^2 dx = 1, \ \mu^{\frac{p}{2}} V(u) \geq \min (C_1 \mu^{2 \beta +1}, \ C_2 g^{-\frac{p}{2}} \mu^{(2 \beta +1)\frac{p}{2}} ) \right\},
    \label{U}
    \nonumber
\end{equation}
with $C_1, C_2$ and $g$ as in Proposition \ref{apriori}.
Then we consider the family of functions $f_u$ defined by
\begin{equation}
    f_u(\mu) := E(\sqrt{\mu} u, \mathcal{G}) = \frac{\mu}{2} \int_{\mathcal{G}} |u'|^2 dx - \frac{\mu^{\frac{p}{2}}}{p} V(u) - \frac \mu 2 \int_{\mathcal{K}} w(x) |u|^2 dx,  \hspace{0.2cm} u \in U,
    \nonumber
\end{equation}

\noindent
so that the following  identities hold:
\begin{equation}
    \I_\G(\mu) \ = \ \lim_{n\rightarrow\infty} E (v_n, \G) = \lim_{n\rightarrow\infty} f_{u_n} (\mu),
\end{equation}
thus
$$\I_\G (\mu) = \inf_{u \in U} f_u (\mu).$$
 
 Now,  since $u \in U$ implies $V (u) > 0$,  
\begin{equation}
%    f_u''(\mu) = - c_p \mu^{\frac{p}{2} -2} V_u - \frac{1}{4} \int_{\mathcal{K}} \gamma^2(x) dx < 0, \hspace{0.5cm} (c_p >0)
  f_u''(\mu) = - \frac{p-2}{4} \mu^{\frac{p}{2} -2} V(u)  < 0,
    \nonumber
\end{equation}
which implies that $\I_\G$ is concave. Furthermore,
since 
$f_u$
is uniformly strictly concave, $\mathcal I_{\G}$ is strictly concave too on every interval $[a, b] \subset (0,\infty)$ and also strictly subadditive since $\I_{\mathcal{G}}(0) = 0$. 
}
\end{proof}

Now we analyze the behaviour of the minimizing sequences.
\begin{proposition}
\label{wcompact}
Any minimizing sequence $\{u_n\} \subset {H}^{1}_{\mu}(\mathcal{G})$ for the functional  $E(\cdot, \mathcal{G})$ defined in \eqref{nlsc}
with $2<p<6$, is weakly compact in $H^{1}(\mathcal{G})$.
\end{proposition}  
\begin{proof} Using \eqref{gn1}, the fact that $2 < p < 6$ and a straightforward estimate for the potential 
term, we get the lower bound
%so 
%\begin{align}
%    \frac{1}{2} \|u'_n\|^2_{L^2(\mathcal{G})} - \frac{1}{p} \|u_n\|^p_{L^p(\mathcal{G})} \geq  \frac{1}{2} \|u'_n\|^2_{L^2(\mathcal{G})} - \zeta_p \|u'_n\|^{\frac{p}{2} - 1}_{L^2(\mathcal{G})}, \hspace{1cm} \zeta_p > 0 \nonumber
%\end{align}
\begin{equation}
E (u, \G) \ \geq \     \frac{1}{2} \|u'_n\|^2_{L^2(\mathcal{G})} - \frac {M_p} p \mu^{\frac p 4 + \frac 1 2} \|u'_n\|^{\frac{p}{2} - 1}_{L^2(\mathcal{G})} - \frac \mu 2 \|w\|_{\infty}, 
\nonumber
\end{equation}
so that $\| u'_n \|_{L^2 (\G)}$ is bounded,
otherwise the sequence $E(u_n, \G)$ would diverge 
as $p < 6$ and the sequence $u_n$ could not be minimizing.
Since $\| u_n \|_{L^2 (\G)} = \sqrt \mu,$
the sequence is bounded in $H^1 (\G)$ and then weakly compact
by the Banach-Alaoglu Theorem.
\end{proof}

In the next result we characterise the behaviour of weakly convergent minimizing sequences, to
which we reduce owing to Proposition
\ref{wcompact}.
\begin{proposition}
\label{minseq}
If $\{u_n\}$ is a minimizing sequence  for the functional $E(\cdot, \mathcal{G})$ at mass $\mu$, and $u_n \rightharpoonup u $ in $H^1(\mathcal{G})$, then 
one of the two following cases occurs:\\
(i) $u_n \rightarrow 0 $ in $L^{\infty}_{loc}(\mathcal{G})$ and $u \equiv 0$;   \\
(ii) $u \in H^1_{\mu}(\mathcal{G})$ is a minimizer and $u_n \rightarrow u$ strongly in $H^1(\mathcal{G}) \cap L^p(\mathcal{G})$. 
\end{proposition} 
\begin{proof}
%The boundedness of $\{u_n\}$ follows from Proposition 3.2.
The sequence $u_n$ converges weakly to $u$ in $L^2(\mathcal{G})$ too, so let \begin{equation}
    m = \mu - \|u\|^2_{L^2(\mathcal{G})} \in [0,\mu]
    \nonumber
\end{equation}
be the loss of mass in the limit. 
%Into the concentration-compactness principle, either $m = \mu$, $0<m<\mu$, or $m = 0$. 
If $m = \mu$, then case $(i)$
occurs because $u_n \rightarrow 0$  strongly in $L^{\infty}_{loc}(\mathcal{G})$ and so $u \equiv 0$. 

Furthermore, following \cite{Ast16}, we can state that the case $0<m<\mu$ never occurs. Indeed, 
first we note that
from  the fact that $u_n \rightharpoonup u$ in $L^2(\mathcal{G})$ one has $$  \| u_n - u \|^2_{L^2(\mathcal{G})} \ = \
\| u_n \|^2_{L^2 (\G)}
+ \| u \|^2_{L^2 (\G)} -
2 \Re (u_n, u) \ \to \ m.
$$

Moreover,
using Brezis-Lieb's lemma \cite{BL} one gets 
\textbf{\begin{equation}
    \frac{1}{p}\int_{\mathcal{G}} |u_n|^p dx - \frac{1}{p}\int_{\mathcal{G}} |u_n - u|^p dx - \frac{1}{p}\int_{\mathcal{G}} |u|^p dx = o(1),
    \nonumber
\end{equation}}
\begin{equation}
    \frac{1}{2} \int_{\mathcal{G}} |u'_n|^2 dx - \frac{1}{2} \int_{\mathcal{G}} |u'_n - u'|^2 dx - \frac{1}{2} \int_{\mathcal{G}} |u'|^2 dx = o(1) \nonumber ,
\end{equation}
\begin{equation}
\frac{1}{2} \int_{\mathcal{K}} w(x) |u_n|^2 dx -  \frac{1}{2} \int_{\mathcal{K}} w(x) |u_n - u|^2 dx - \frac{1}{2} \int_{\mathcal{K}} w(x) |u|^2 dx  = o(1) .\nonumber
\end{equation}

Therefore 
\begin{equation}
    E(u_n, \mathcal{G}) = E(u_n-u, \mathcal{G}) + E(u, \mathcal{G}) + o(1)
    \nonumber
\end{equation}
as $n \rightarrow \infty$. Now, 
 since by concavity $\I_{\mathcal{G}}$ is continuous, one obtains 
\begin{equation}
     \I_{\mathcal{G}}(\mu) \ \geq \ \I_{\mathcal{G}}(m) + E(u,{\mathcal{G}}) \ \geq \ \I_{\mathcal{G}}(m) + \I_{\mathcal{G}}(\|u\|^2_{L^2(\mathcal{G})}).
   \nonumber
\end{equation}
\indent From Proposition \ref{subadditive} the function $\I_{\mathcal{G}}$ is strictly subadditive, therefore if $0<m<\mu$ then
\begin{equation}
     \I_{\mathcal{G}}(\mu) < \I_{\mathcal{G}}(\mu -\|u\|^2_{L^2(\mathcal{G})}) + \I_{\mathcal{G}}(\|u\|^2_{L^2(\mathcal{G})})
     \nonumber
\end{equation}
which is a contradiction. \\
\indent If $m = 0$ then $u_n \rightarrow u$ strongly in $L^2(\mathcal{G})$. Now, since $p>2$ and $\|u_n\|_{L^{\infty}} \leq C$, from 
\begin{equation} \label{convV}
\| u_n - u \|_{L^p (\G)}^p \leq 
\| u_n - u \|_{L^\infty (\G)}^{p-2}
\| u_n - u \|_{L^2 (\G)}^2
\leq C \| u_n - u \|_{L^2 (\G)}^2 \to 0,
\end{equation}
one gets $u_n \to u$ strongly in $L^p (\G)$ too.

\noindent
\indent Furthermore
\begin{equation} \label{convW}\begin{split}
\int_{\mK} w (x) \left||u_n (x)|^2 - |u (x)|^2
\right| dx \ \leq & \ C 
\int_{\mK} w (x) |u_n (x) - u (x)|^2 dx  \\ \leq & \ C
\| w \|_\infty 
\| u_n - u \|^2_{L^2 (\G)} \to 0
\end{split}
\end{equation}
so $W(u_n)$ converges to $W(u)$.

%\noindent
Then, from Eq \eqref{convV}, Eq \eqref{convW}, and the fact
that $u_n$ is a minimizing sequence,
\begin{equation} \begin{split}
\| u_n' \|_{L^2 (\G)}^2 \ = & \ E (u_n, \G) + \frac{1}{p}
\| u_n \|_{L^p (\G)}^p + \frac 1 2 \int_{\mK} 
w (x) |u_n (x)|^2 dx \\ \to &  \ \I (\mu) + 
+ \frac{1}{p}
\| u \|_{L^p (\G)}^p + \frac 1 2 \int_{\mK} 
w (x) |u (x)|^2 dx \\
\leq & \  E (u, \G) + \frac{1}{p}
\| u \|_{L^p (\G)}^p + \frac 1 2 \int_{\mK} 
w (x) |u (x)|^2 dx \ = \ \| u' \|^2_{L^2 (\G)}.
\end{split}
    \end{equation}
where the last passage exploits $u \in H^1_\mu (\G).$ 
It follows 
\begin{equation}
    \label{kinu}
  \lim  \| u_n ' \|^2_{L^2 (\G)} \ \leq \
    \| u' \|^2_{L^2 (\G)}.
\end{equation}
\indent On the other hand, since $u'$ is the weak limit in
$L^2 (\G)$ of $u_n'$, it must be
\begin{equation}
    \label{kinubis}
  \liminf  \| u_n ' \|^2_{L^2 (\G)} \ \geq \
    \| u' \|^2_{L^2 (\G)}.
\end{equation}

From Eqs \eqref{kinu} and \eqref{kinubis} one concludes
$  \lim \| u_n ' \|^2_{L^2 (\G)} \ = \
    \| u' \|^2_{L^2 (\G)},$
so $u_n$ converges to $u$ strongly in $H^1 (\G)$,
and therefore $E (u, \G) = \lim E (u_n, \G)$ and $u$
is a Ground State for $E (\cdot, \G)$ at mass $\mu.$
The proof is complete.
\end{proof}

We are ready to prove
Lemma \ref{genex}.

\smallskip 

\noindent
{\em Proof of Lemma \ref{genex}.}
%{\rosso "more detail"?}
Given a minimizing sequence $\{ u_n \} \subset H^1_{\mu}(\mathcal{G})$ we need to exclude case $(i)$ of Proposition \ref{minseq}. Assume $(i)$. Then  $u_n \rightarrow 0$ in $L^{\infty}_{loc}(\mathcal{G})$, that implies $u_n \rightarrow 0$  in $L^{\infty}(\mathcal{K})$. Since the potential is supported on $\mathcal{K}$, then $E(u_n,\mathcal{G}) = E_{\NLS}(u_n, \mathcal{G}) + o(1)$ and one  can repeat the argument in Theorem 3.3 of  \cite{Ast16}. Such argument shows that such a minimizing sequence leads to an energy level not lower than the threshold $- \theta_p \mu^{2 \beta + 1}$.  Therefore, if there exists $v \in H^1_{\mu}(\mathcal{G})$ such that the condition (\ref{criterion}) is satisfied, then case $(i)$ of Proposition \ref{minseq} is ruled out, so case $(ii)$ is verified. Thus we can assume the existence of a minimizing sequence $\{u_n\}$ that strongly converges to $u$ in $H^1 \cap L^p(\mathcal{G})$ and $u$ is a Ground State. 
$\qedsymbol$

%Owing to this theorem we can state that if there is a competitor function in $H^1_{\mu}(\mathcal{G})$ such that its energy is lower than the energy of the soliton, then the graph $\mathcal{G}$ admits a ground state with mass $\mu$. Therefore, finding such competitoriklik provides a tool for proving the existence of these states. 

\section{Existence}

%We focus on graphs $\mathcal{G}$ with $n$ halflines. 
It is well known that a graph $\G$ %ground states of $E_{\NLS} (\cdot, \G)$ at
does not support ground states for $E_{\NLS} (\cdot, \G)$ at
any mass, provided that it satisfies the so-called Assumption H (see \cite{Ast15}), i.e., if every point of $\mathcal{G}$ belongs to a trail that contains two halflines. This is the case, for instance, of the graph in Figure (\ref{fig:2b}), called the 2-bridge.

\begin{figure}[h!]
 \centering
 \includegraphics[height = 4 cm ]{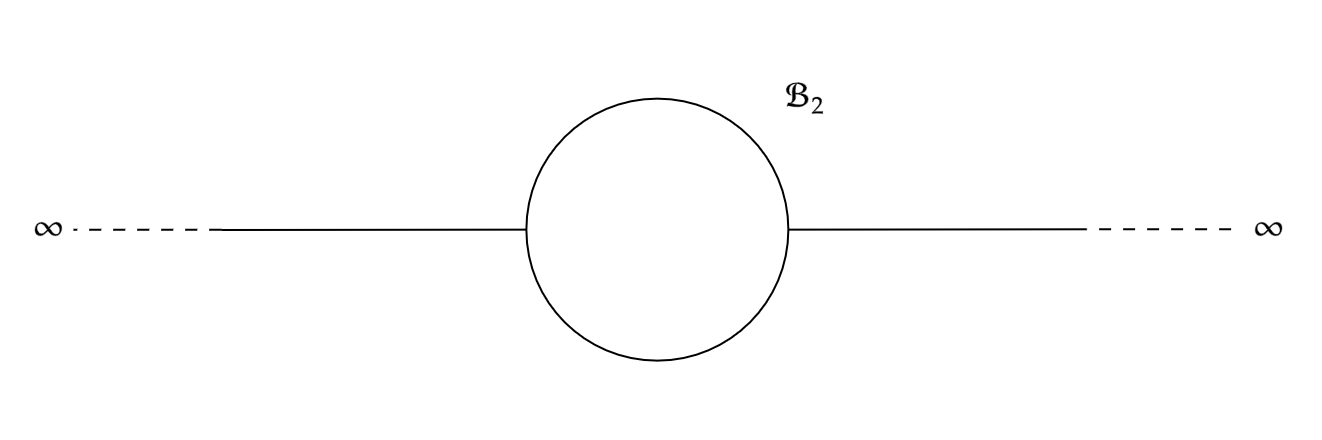}
        \caption{The 2-bridge graph satisfies the assumption (H) that prevents the
        existence of ground states for $E_{\NLS}$ at every mass.}
        \label{fig:2b}
\end{figure}

On the other hand, the presence
of an attractive potential  can make it energetically convenient to concentrate the mass in the support of the potential and therefore, at least for some values of the mass, ground states can exist. 
In the present section we
show that this is indeed the case 
if the mass is large or
small enough. We achieve the results by directly applying
Lemma \ref{genex}, i.e., exhibiting 
a function of mass $\mu$ whose energy is
below the threshold $- \theta_p \mu^{2 \beta + 1}$.

\subsection{Existence for large mass}
Here we give the first result on existence of ground states.
\begin{theorem}
\label{large mass}
Let $\mathcal{G}$ be a graph with $n \geq 1$ halflines and let 
 $w \geq 0$ be a continuous,
 non identically vanishing function supported on the compact core
$\mathcal{K}$ of $\G$.

If $\mu$ is large enough, then there exists a Ground State for $E(\cdot, \G)$ at mass $\mu$.
\end{theorem}

\begin{proof}
%{\rosso Questa descrizione è imprecisa. Bisogna partire da un edge dove $w$ è non nullo e stringere il solitone lì. Cerca di descrivere questa procedura in modo preciso.}
%In the large mass limit, a good ground state candidate is a function supported in the compact part of the graph where $w(x)$ is defined. 
%We construct this function by cutting the soliton symmetrically and rescaling its mass. We denote this function by $u$.
%Hence we aim at proving that 
%\begin{equation}
 %   E(u, \mathcal{G}) < E(\phi_{\mu}, \mathbb{R}) .
%\end{equation}
%where $\phi_{\mu}$ is the usual soliton with mass $\mu$ that describe the ground state of the energy functional in the line. 

\noindent
Consider a point $\bar x$ in $\mathcal K$ such that $\bar x$ is
not a vertex of $\G$ and $w (\bar x) > 0.$
Such a point exists, since otherwise by continuity $w$ would be zero  in all vertices too, and then identically zero. With a slight abuse of notation, we denote by $\bar x$ 
the coordinate of the point $\bar x$ as an element of the interval $I_{\bar e}$ that represents the edge $\bar e$ in which $\bar x$ lies. By continuity, there is $\ell > 0$ such that the interval $[\bar x - \ell/2, \bar x + \ell/2]$ belongs to $\bar e$ and $w > 0$
in $[\bar x - \ell/2, \bar x + \ell/2]$. We denote $\kappa: = \min_{[\bar x - \ell/2, \bar x + \ell/2]} w (x)$, so that $\kappa > 0.$

Consider now the family of functions $v_\mu \in H^1 (\R)$ defined 
 in the following
way:
\begin{equation}
    v_\mu (x) = (\phi_\mu (x - \bar x) - \phi_\mu (\ell/2))
    \chi_{[\bar x - \ell/2, \bar x + \ell/2] } (x) , 
    \nonumber
\end{equation}
where $\chi_A$ denotes the characteristic function of the real subset $A.$

\begin{figure}[h!]
 \centering
 \includegraphics[height = 4 cm ]{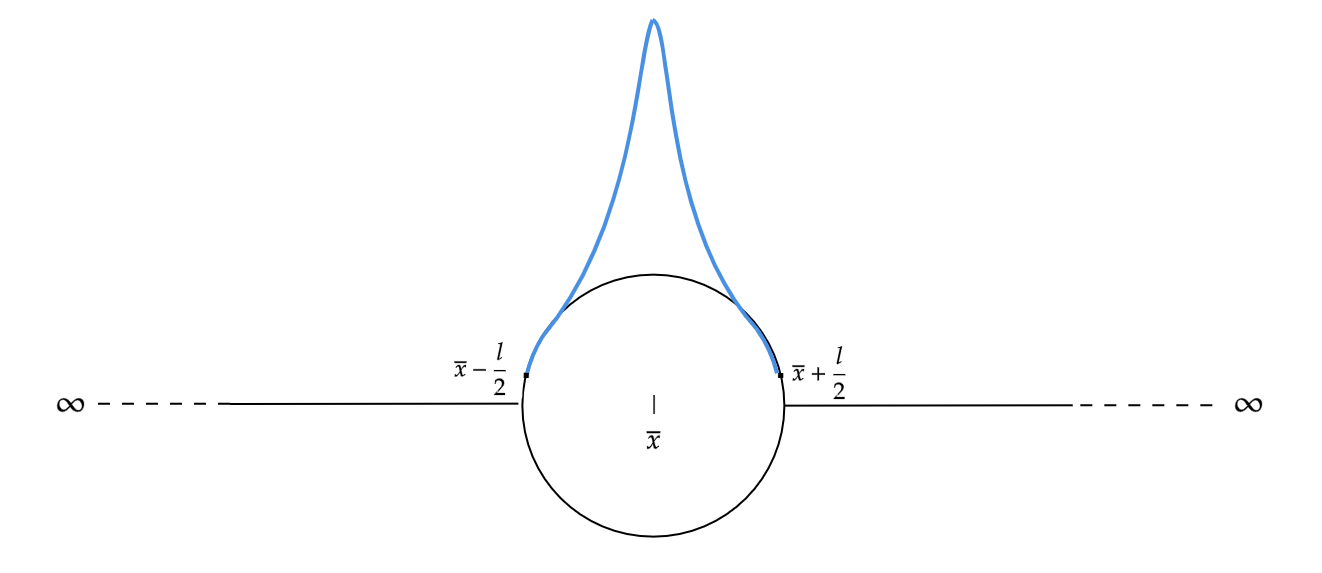}
        \caption{A pictorial representation of the function $\widetilde f_\mu$  on the 2-bridge graph.}
        \label{fig:big mass}
\end{figure}

%\noindent

One has
\begin{equation} \label{massv} \begin{split}
    \int_{\mathcal{\R}} |v_\mu|^2 dx  & =  \int_{-\frac{\ell}{2}}^{\frac{\ell}{2}} | \phi_{\mu} - \phi_{\mu}(\frac{\ell}{2})|^2 dx = 2 \int_{0}^{\frac{\ell}{2}} | \phi_{\mu} - \phi_{\mu}(\frac{\ell}{2})|^2 dx  \\
    & = 2 \int_{0}^{\frac{\ell}{2}}  \phi_{\mu}(x)^2 dx + 2 \int_{0}^{\frac{\ell}{2}} \phi_{\mu}(\frac{\ell}{2})^2dx - 4 \int_{0}^{\frac{\ell}{2}} \phi_{\mu}(x) \phi_{\mu}(\frac{\ell}{2}) dx  \\
  &  = \mu - 2 \int_{\frac{\ell}{2}}^{\infty} \phi_{\mu}(x)^2 dx+ \ell \phi_{\mu}(\frac{\ell}{2})^2 - 4 \phi_{\mu}(\frac{\ell}{2}) \int_{0}^{\frac{\ell}{2}} \phi_{\mu}(x) dx.
\end{split} \end{equation}

%If we use the large mass approximation $\mu \rightarrow \infty$, then
\noindent

From the explicit form of the soliton \eqref{soliton},
since $\ell > 0$ one immediately has 
$\phi_\mu (\ell/2) \leq  C \mu^\alpha e^{-c \mu^\beta}$. Moreover, since $2 \alpha - \beta = 1,$
\begin{equation} \label{coda}
 \int_{\frac{\ell}{2}}^{\infty} \phi_{\mu}(x)^2 dx \leq \phi_\mu (\frac \ell 2)   \int_0^{\infty} \phi_{\mu}(x) dx \leq C \mu  e^{-c \mu^\beta},
\end{equation}
where $C$ and $c$ are positive constants independent of
$\mu$.
Then,  recalling that $v_\mu$ is a cut and lowered version of $\phi_\mu$, and that in \eqref{massv} both
$\phi_\mu (\ell/ 2) $ and $\phi_\mu^2 (\ell/ 2) $ appear,
we conclude
\begin{equation}
    \label{massvmu}
   0 \leq \ \mu - \| v_\mu \|^2_{L^2 (\R)}
    \ \leq  \ C ( \mu^{2 \alpha} + \mu)  e^{-c \mu^\beta}.
\end{equation}

Furthermore, 
\begin{equation} 
    \label{v-phi}
    \| v_\mu - \phi_\mu (\cdot - \bar x) \|^2_{L^2 (\R)} = 2 \int_{0}^{\ell/2}
    \phi_\mu^2 (\ell/2) dx + 2
    \int_{\ell/2}^\infty \phi_\mu^2 (x) \, dx
    \
   \leq \ C (\mu^{2\alpha} + \mu) e^{-c \mu^\beta}  
    ,
\end{equation}
and
\begin{equation} 
    \label{v'-phi'}
    \| v_\mu' - \phi_\mu' (\cdot - \bar x) \|^2_{L^2 (\R)} =  2
    \int_{\ell/2}^\infty (\phi_\mu')^2 (x) \, dx
     \leq  C \mu^{2 \beta}
    \int_{\ell/2}^\infty 
    \phi_\mu^2 (x) \, dx  \leq  C \mu^{2 \beta + 1} e^{-c \mu^\beta}
 \end{equation}
where the last estimate is obtained from
inequality
$$ |\phi_\mu' (x) | \ = \ C \mu^{\alpha + \beta}
\frac{| \sinh (\mu^ \beta c_p x)|}{ \cosh^{\frac \alpha \beta +1} (\mu^\beta c_p x)}
\leq C \mu^{\beta} \phi_\mu (x)$$
and by Eq \eqref{coda}.

\noindent

Now we define the function
$f_\mu = \frac{\sqrt \mu}{\| v_\mu \|_{L^2 (\R)}} v_\mu$ and compute
\begin{equation} \label{fprime}
    \begin{split}
    \| f_\mu - v_\mu \|_{L^2 (\R)} & =
     {\sqrt \mu - \| v_\mu \|_{L^2 (\R)}} \ = \
     \frac{\mu - \| v_\mu \|_{L^2 (\R)}^2} 
     { \sqrt \mu + \| v_\mu \|_{L^2 (\R)}}  
     \ \leq \ C (\sqrt \mu + \mu^{2\alpha - \frac 1 2})
     e^{-c \mu^\beta } \\
    \| f_\mu' - v_\mu' \|_{L^2 (\R)} & 
    \leq  ({\sqrt \mu - \| v_\mu \|_{L^2 (\R)}} )
    \frac{\| v_\mu' \|_{L^2 (\R)}}{\| v_\mu \|_{L^2 (\R)}} 
      \leq \ C ( \sqrt \mu + \mu^{2\alpha - \frac 1 2}) e^{-c \mu^\beta} \frac{\| v_\mu' \|_{L^2 (\R)}}{\| v_\mu \|_{L^2 (\R)}}.
      \end{split} \end{equation}
      From \eqref{coda}
      $$ \| v_\mu \|_{L^2 (\R)} \geq \sqrt{\mu - C (\mu + \mu^{2 \alpha}) e^{-c \mu^\beta}}$$
      while, since $\| \phi_\mu' \|_{L^2 (\R)} \leq C \mu^{\beta + \frac 1 2}$, from \eqref{v'-phi'} one gets
      $$\| v_\mu' \|_{L^2 (\R)} 
      \ \leq \ \| \phi_\mu' \|_{L^2 (\R)} + C \mu^{\beta + \frac 1 2} e^{-c \mu^\beta} \ \leq \ C \mu^{\beta + \frac 1 2} 
      (1 + e^{-c \mu^\beta}) 
      .
      $$
     
      Thus, for the second inequality in \eqref{fprime} one gets
\begin{equation} \label{fsecond}
    \begin{split}
    \| f_\mu' - v_\mu' \|_{L^2 (\R)} & 
      \leq \ C ( \sqrt \mu+ \mu^{2\alpha - \frac 1 2}) e^{-c \mu^\beta} \frac{\mu^{\beta + \frac 1 2}  (1 +  e^{-c \mu^\beta})}
      {\sqrt{\mu - C (\mu + \mu^{2 \alpha}) e^{-c \mu^\beta}}},
      \end{split} \end{equation}
      
      %\ \sim \ C \mu^{2 \alpha} e^{-c \mu^\beta} \frac{\| \phi_\mu' \|_{L^2 (\R)}}{\| \phi_\mu \|_{L^2 (\R)}} \\ & \sim C \mu^{2 \alpha + \beta} e^{-c \mu^\beta},
    %\end{split}
%\end{equation}
\noindent
where we used Eqs \eqref{coda}\text{–}\eqref{v'-phi'}. %\eqref{massvmu}, \eqref{v-phi}, and 
From Eqs \eqref{fprime} and \eqref{fsecond} one then concludes that
both $\| f_\mu - v_\mu \|_{L^2 (\R)}$ and $\| f'_\mu - v'_\mu \|$ vanish as $\mu$ goes to infinity.

\noindent

Furthermore, from Eqs \eqref{gn1} and \eqref{gn2}
\begin{equation}
     \| f_\mu - v_\mu \|_{L^p (\R)}^p \leq C   \| f_\mu' - v_\mu' \|_{L^2 (\R)}^{\frac{p}{2}- 1} 
      \| f_\mu - v_\mu \|_{L^2 (\R)}^{\frac{p}{2}+ 1} 
     \to 0,
     \quad \mu \to \infty.
\label{lp}
\end{equation}

Let us introduce the function $\widetilde f_\mu,$ defined as $f_\mu$ on the edge $\bar e$ and zero 
on all other edges. Obviously, 
$\widetilde f_\mu$ belongs to $H^1 (\G)$
and its $L^2 (\G), \, L^p (\G),$ and $H^1 (\G)$ norms are the same as the corresponding ones of $f_\mu$
as a function on $\R.$ 
Now, from Eqs \eqref{fprime} and \eqref{v'-phi'} one has
\begin{equation} \begin{split}
    \label{cinetica}
  \left|  \| \widetilde f_\mu'\|_{L^2 (\G)} - \| \phi_\mu'\|_{L^2 (\R)} \right| & =  \left| \|  f_\mu'\|_{L^2 (\R)} - \| \phi_\mu'\|_{L^2 (\R)} \right| 
  \leq 
  \| f_\mu' - \phi_\mu' \|_{L^2 (\R)}
\\ &  \leq \| f_\mu'
    -  v_\mu'\|_{L^2 (\R)}  + \| v_\mu'- \phi_\mu'\|_{L^2 (\R)}
    \to 0, \quad \mu \to \infty, 
    \end{split}
\end{equation}
and analogously from Eq \eqref{lp}
\begin{equation}
    \label{potenziale}
    \| \widetilde f_\mu\|_{L^p (\G)}^p - \| \phi_\mu\|_{L^p (\R)}^p   \to 0, \quad \mu \to \infty.
\end{equation}

By Eqs \eqref{cinetica} and {\eqref{potenziale}}
\begin{equation} \begin{split}
    E (\widetilde f_\mu, \G) - E_{\NLS} ( \phi_\mu, \R) \  & 
= \   E_{\NLS} (\widetilde f_\mu, \G) - E_{\NLS} ( \phi_\mu, \R)
- \frac 1 2 \int_\G w (x) | \widetilde f_\mu (x)|^2 dx
    \\
&    \leq \  - \frac \kappa 2 \int_{-\ell/2}^{-\ell/2}
    \phi_\mu^2 (x) \, dx + o (1)
    \\ &  \leq \  - \frac \kappa 2  \mu + o (1)
    , \quad \mu \to \infty,
    \end{split}
    \nonumber
\end{equation}
so 
$E (\widetilde f_\mu, \G) < E_{\NLS} ( \phi_\mu, \R)$ for $\mu$ large enough and by Lemma \ref{genex} the proof is complete.
\end{proof}

\subsection{Existence for small mass}
Here we give the second result on existence of ground states.
\begin{theorem}
\label{small mass}
%Let $\mathcal{G}$ be a graph with $n$ halflines $\mathcal{H}_n$ and a compact $\mathcal{K}$ where the potential $w(x)$ is defined. If $\mu > 0 $ is small enough, then a ground state always exists. 
Let $\mathcal{G}$ be a graph with $n \geq 1$ infinite edges and let 
 $w \geq 0$ be a continuous,
 non identically vanishing function supported on the compact core
$\mathcal{K}$ of $\G$.

If $\mu$ is small enough,
then there exists a Ground State for $E(\cdot, \G)$ at mass $\mu$.
\end{theorem}
\begin{proof}
Let $\mu > 0$.
We define the function $u_\mu$ as follows:
\begin{equation}
u_\mu (x) = 
    \begin{cases}
       \phi_{m}  \hspace{9.5mm} $if$ \hspace{3mm} x \in \mathcal{H}_i, \ i = 1, \dots, n \\
       \phi_{m}(0)   \hspace{5mm} $if$ \hspace{3mm} x \in \mathcal{K},
    \end{cases}
    \nonumber
\end{equation}
where $\mathcal{H}_i$ represents the halfline associated with the index $i$.
\begin{figure}[h!]
 \centering
 \includegraphics[height = 4 cm ]{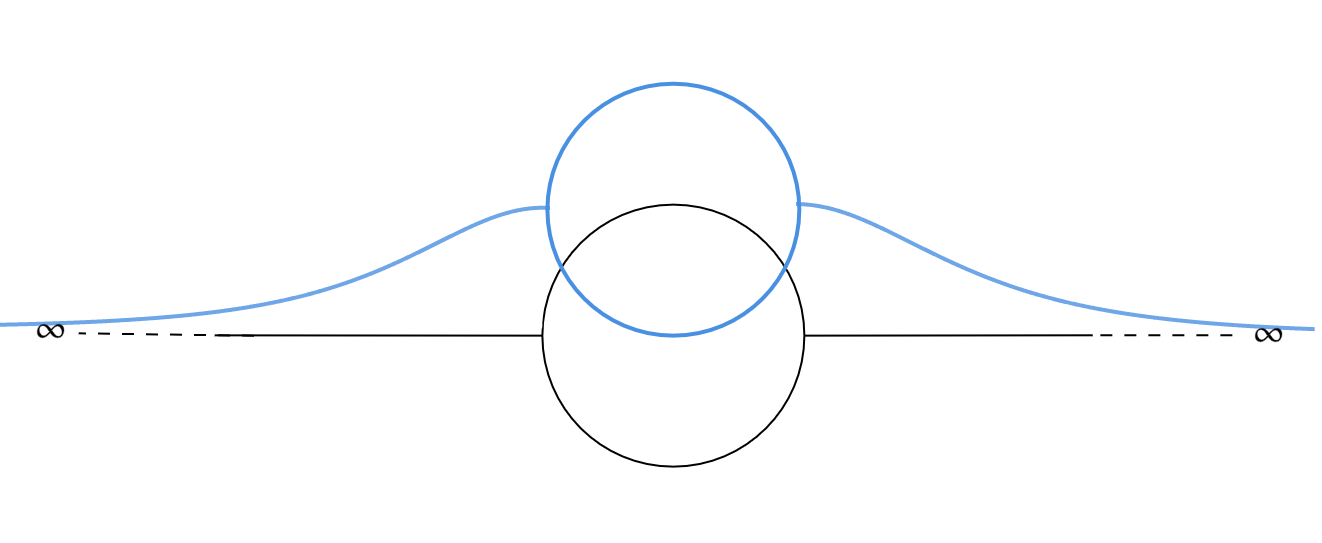}
        \caption{A representation of the function $u_\mu$ on the two-bridge graph.}
        \label{fig:small mass}
\end{figure}
The parameter $m$ is uniquely determined by imposing $\|u_\mu\|^2_{L^2 (\R)} = \mu$, namely by the identity
\begin{equation}
    \mu = \int_{\mathcal{G}} |u_\mu|^2 dx = \frac n 2 m + |\phi_{m}(0)|^2 |\mathcal{K}| = \frac n 2 m  + C_p^2 |\mathcal{K}| m^{2 \alpha},
    \label{ehs}
    \nonumber
\end{equation}
where $|\mathcal{K}|$ is the total length of $\mathcal{K}$, while the energy of $u_\mu$ reads
\begin{equation}
%    E^c(u) = - N 2^{2 \beta} \theta_p m^{2 \beta +1} -  \frac{\gamma^2}{4} C m^{2 \alpha} l
E(u_\mu,\G) = - \frac{n}{2} \theta_p m^{2 \beta + 1} - \frac{1}{p} C_p^p m^{p \alpha} |\mathcal{K}| -  C_p^2 m^{2 \alpha} \frac{|\mathcal{K}|} 2 \int_{\mathcal{K}}w(x)dx .
    \nonumber
\end{equation}

In order to be less energetic than the soliton on the line with the same mass, $u_\mu$ must satisfy the
condition 
\begin{equation} \label{smallineq}
 - \frac{n}{2} \theta_p m^{2 \beta + 1} - \frac{1}{p} C_p^p |\mathcal{K}| m^{p \alpha}  - C_p^2 \frac{|\mathcal{K}|} 2 m^{2 \alpha} \int_{\mathcal{K}}w(x)dx  < - \theta_p ( \frac{n}{2} m  + C_p^2 {|\mathcal{K}|}  m^{2 \alpha} )^{2 \beta + 1}. 
\end{equation}

%Since $2 \alpha = \beta +1$ and $p \alpha = 3 \beta + 1$, we can use the following set of inequalities:

%($2 \alpha = \beta + 1$, $p \alpha = 3 \beta + 1$)

\noindent

Since $2 \alpha$ is the smallest exponent in inequality Eq \eqref{smallineq}, it
turns out that the inequality is satisfied if $\mu$ is small enough, thus by Lemma \ref{genex} a Ground State exists and the theorem is proved.

%\begin{equation}
%C_p^2 L \int_{\mathcal{K}}w(x)dx > ((\frac{n}{2})^{2 \beta + 1} - \frac{n}{2}) \theta_p m^{\beta} + o(m^{\beta}).
%\end{equation}
\end{proof}

\section{Nonexistence}
Here we extend the nonexistence result given in {Theorem 5.1} of \cite{Ast16} to the case
of the presence of a weak, attractive, compactly supported potential and give sufficient conditions for the nonexistence
of ground states in some intervals
of the values of the mass. 

%We preliminary recall a basic 
%estimate for functions defined
%on a compact set $\mathcal K \subset \G$ (\cite{EstimatesHaesler}):
%\begin{equation}
 %   \|u\|_{L^{\infty}(\mathcal{K})} \leq |\mathcal{K}|^{-\frac{1}{2}} \|u\|_{L^2(\mathcal{K})} + \text{diam}(\mathcal{K})^{\frac{1}{2}} \|u'\|_{L^2(\mathcal{K})}, \hspace{0.5cm} \forall u \in H^1(\mathcal{K})
  %  \label{sobolev}
%\end{equation}
%where $|\mathcal K|$ is the total length of $\mathcal{K}$ and $\text{diam}(\mathcal{K})$ is its diameter.  \\
\begin{theorem}
\label{th:nonex}
Let $\G$ be a graph with $n \geq 1$ infinite edges
and let 
 $w \geq 0$ be a continuous,
 non identically vanishing function supported on the compact core
$\mathcal{K}$ of $  \G$.
Furthermore
denote by $|\mathcal K|$  the total length of $\mathcal{K}$ and by $\rm{diam}(\mathcal{K})$ its diameter, i.e. the maximal distance between any pair of points of $\mathcal{K}$.  

Then there exists a number $\epsilon > 0$, that depends on $p$ only, such that, if 
%the 
%conditions
\begin{equation} \begin{split}
%\mu^{2\beta} > & \frac{ 2 |\mathcal K| \| w \|_\infty}{\theta_p} , \quad
    \max %\begin{Bmatrix}
    \left( \mu^{\beta} {\rm{diam}}(\mathcal{K}), \frac{1}{\mu^{\beta} |\mathcal{K}|}  %\end{Bmatrix} 
    ,
    \frac{ \| w \|_\infty}{\mu^{2 \beta}}
    \right)
    <  \epsilon
    \label{theor}
    \end{split}
\end{equation}
are satisfied,
then the functional $E (\cdot,\mathcal{G})$ defined in \eqref{nlsc} has no Ground State at mass $\mu$. 
\end{theorem}

\begin{proof}  %The proof of points (ii) and (iii) does not present significant difference with respect to that of Theorem 5.1 in \cite{Ast16}.
 We proceed by contradiction, thus we consider $\mu > 0$ that 
satisfies the condition Eq \eqref{theor} and suppose that there exists 
a Ground State $u$ at mass $\mu$ for $E (\cdot, \G)$. Due to the invariance of $E (\cdot, \G)$ under multiplication by a phase, we can assume without loss of generality that $u$ is real and nonnegative (see e.g., Sec. 1 in \cite{Ast15}).

First we assume that $\G$ contains only one halfline, i.e.  $n = 1$.
Taking $\epsilon < \theta_p / 2$, condition Eq \eqref{theor}
guarantees that inequality Eq \eqref{smallW} holds, so that estimates Eq \eqref{largemu} are valid. 
 It is therefore possible to
follow the proof of Theorem 5.1 in \cite{Ast16} replacing $E(\phi_\mu, \mathbb{R})$ by $E_{\NLS}(\phi_\mu, \mathbb{R})$ up to the
last inequality, which has to
be rephrased { including the contribution of the potential}, i.e.,
 \begin{equation} \begin{split}
 \label{final}
%\begin{align*}
   & E(u, \mathcal{G}) - E_{\NLS}(\phi_{\mu}, \mathbb{R})   \\
   \geq &
 % & \geq \left(C^{-1} - C \epsilon^{\frac{p-2}{2}}\right) \mu^{2\beta} \int_{\mathcal{K}} |u|^2 dx - C \mu^{\beta} \phi_m(y)^2 + \frac{1}{2} \int_{\mathcal{K}} |u'|^2 dx - \|w\|_{2} \|u\|_{L^{\infty}(\mathcal{K})} \\
    \left(C_{1} - C_2 \epsilon^{\frac{p-2}{2}}\right) \mu^{2\beta} \int_{\mathcal{K}} |u|^2 dx + \frac{1}{2} \int_{\mathcal{K}} |u'|^2 dx - C_3 \mu^{\beta}  \|u\|^2_{L^{\infty}(\mathcal{K})} - \frac 1 2
    \int_{\mathcal K} w (x) | u |^2 dx \\
    \geq & \left(C_{1} - C_2 \epsilon^{\frac{p-2}{2}}\right) \mu^{2\beta} \int_{\mathcal{K}} |u|^2 dx + \frac{1}{2} \int_{\mathcal{K}} |u'|^2 dx - C_3 \mu^{\beta}  \|u\|^2_{L^{\infty}(\mathcal{K})} - \frac 1 2 \| w \|_{\infty} 
    \int_{\mathcal K}  | u |^2 dx \\
    \geq &
    \left(C_{1} - C_2 \epsilon^{\frac{p-2}{2}}
    - \frac \epsilon 2 \right) \mu^{2\beta} \int_{\mathcal{K}} |u|^2 dx + \frac{1}{2} \int_{\mathcal{K}} |u'|^2 dx - C_3 \mu^{\beta}  \|u\|^2_{L^{\infty}(\mathcal{K})}  .
    \end{split}
    \nonumber
     \end{equation}

Using formula (40) in \cite{Ast16} to estimate $\|u\|_{L^{\infty}(\mathcal{K})}$ one finds 
 \begin{equation} \begin{split}
 \label{finalfinal}
%\begin{align*}
   & E(u, \mathcal{G}) - E_{\NLS}(\phi_{\mu}, \mathbb{R})   \\
   \geq & 
   \left(C_{1} - C_2 \epsilon^{\frac{p-2}{2}}
    - \frac \epsilon 2  - 2 C_3 \epsilon \right) \mu^{2\beta} \int_{\mathcal{K}} |u|^2 dx + \left(\frac{1}{2} - 2 C_3 \epsilon \right)
    \int_{\mathcal{K}} |u'|^2 dx 
\\ > & \ 0,
\end{split}
\nonumber
\end{equation}
uniformly in $\mu$ and provided that $\epsilon$ is small enough. By Remark \ref{gs} this contradicts the fact that $u$ is a Ground State of $E(\cdot, \G)$, and the proof is complete for the case $n=1$. 

Suppose now  $n > 1$. % then
%we proceed by supposing that 
%$E(\cdot, \G)$ admits a Ground State at a given mass $\mu$ that satisfies \eqref{theor}.
Let $\mathcal H_{\widetilde \i}$ be the halfline in which  $u$ attains
$\max_{\G \symbol{92}  \mathcal K} u$, and let us call $\widetilde x$ the corresponding maximum point on ${\mathcal H}_{\widetilde \i}.$ 
Moreover, let us define the halfline $\wmH$ as the subset of $\mH_{\wi}$ corresponding to the coordinate interval $[ \widetilde x, + \infty).$
It is convenient to use on $\wmH$ the coordinate system inherited by $\mH_{\wi}$ ranging from $\widetilde x$ to $+\infty.$

For any $i \neq \widetilde \i$   let us set
\begin{equation} \label{reconnect}
\wy_i = \min \{ x \in [\widetilde x, + \infty),  \,
 u_{\widetilde \i} (x) = u_i (0)\},
\end{equation}
which is well-defined since 
by definition of $\widetilde x,$ for every $i \neq \widetilde \i$ one has $u_i (0) \leq u_{\widetilde {\i}} (\widetilde x)$, thus by continuity of
$u_{\widetilde {\i}}$ there exists at least a point $\widetilde z \in 
\wmH$ such that $u_{\widetilde {\i}} (\widetilde z) = u_i (0).$ The symbol $\wy_i$ denotes then the minimum of such $\widetilde z$'s. 

Now for every $i \neq \wi$ we attach the origin of $\mH_i$ to the  point of coordinate $\wy_i$ in the halfine   $\wmH$. We obtain in this way a graph $\whG$, made of one halfline ($\wmH$), to which are attached by their origins $n-1$ other halflines ($\mH_i, \ i \neq \wi$).

Let us consider the function $\widehat u : \widehat \G \to {\R},$ which is made of the restrictions of $u$ to the halflines
that constitute the graph $\widehat \G$. In symbols, $\widehat u =
( \widehat u_{\wi},\widehat u_i)_{i \neq {\wi}},$ with
$
\widehat u_{\wi} (x) =  u_{\wi} (x)$ for every $x \in [\widetilde x, + \infty)$ and
$\widehat u_i (x) = u_{i} (x)$ for every $x \in [0, + \infty).$ 
Since the restriction of $\widehat u$ to every halfline is in $H^1$,
and since by definition of the points $\wy_i$ the function $\widehat u$ is continuous at the vertices of $\whG,$ it follows
that $\widehat u \in H^1 (\whG)$. Moreover, $\whG$ is connected and therefore one can apply Proposition \ref{rearrangements} and then define
the monotone rearrangement $u^*$ of $\widehat u$, that is defined on $[0, + \infty).$

Now we define the graph $\G'$ as the original graph $\G$, but with $\mH_{\wi}$ as the only 
halfline, i.e.,
$$ \G' = \G \symbol{92} 
(\cup_{i \neq \widetilde \i} \mathcal H_i) = \mathcal K \cup \mathcal H_{\widetilde \i},$$
and construct on it the function $v : \G'\to \R$ as

\begin{equation}
\label{utilde}
 v (x) : = \left\{
\begin{array}{cc}
    u (x), & x \in \mathcal K  \\
     u_{\widetilde \i} (x), & x \in [0, \widetilde x] \subset \mathcal H_{\widetilde \i} \\
     u^*(x - \widetilde x), & x \in (\widetilde x, + \infty) \subset \mathcal{H}_{\widetilde \i}.
\end{array}
\right.
\nonumber
\end{equation}
%where $u^*$ denotes the monotone rearrangement on $[0, + \infty)$ of the
%functions $u_i$ with $ i \neq \widetilde \i$ and $u_{\widetilde \i} (\cdot - \widetilde x)$.

From Proposition \ref{rearrangements} the monotone rearrangement preserves the $L^p$-norms (see identity Eq \eqref{equimeas}), 
then
\begin{equation}
    \label{rearlp} \begin{split}
    \| v \|^r_{L^r (\G')} = & \| v \|^r_{L^r (\mathcal K)}
+  \| v \|^r_{L^r (\mathcal H_{\wi})} 
%\\
%= &
= \| v \|^r_{L^r (\mathcal K)}
+  \int_0^{\widetilde x} | v_{\wi} |^r dx +  \int_{\widetilde x}^{+\infty} | v_{\wi} |^r dx \\ = &
\| u \|^r_{L^r (\mathcal K)}
+  \int_0^{\widetilde x} | u_{\wi} |^r dx +  \int_{\widetilde x}^{+\infty} | u^* (x- \widetilde x) |^r dx 
\\ = &
\| u \|^r_{L^r (\mathcal K)}
+  \int_0^{\widetilde x} | u_{\wi} |^r dx +  
\sum_{i \neq \wi} \| u \|^r_{L^r (\mH_i)} +
\int_{\widetilde x}^{+\infty} | u_{\wi} |^r dx 
\\ = &
\| u \|^r_{L^r (\mathcal K)},
\end{split}
\end{equation}
for every $r \in [1, + \infty].$ In particular, for $r = 2$ one has 
$ \| v \|^2_{L^2 (\G')} = \mu.$

Now, due to Eq \eqref{rearlp} and to $u \equiv v$ on $\mathcal K$, 
the difference  $E (u, \G) - E (v, \G')$ reduces to $T(u) -T(v)$
outside $\mathcal K,$ i.e., on the halflines only. Thus
\begin{equation}
    \label{final} \begin{split} &
E (u, \G) - E (v, \G') \\ = &  
 \sum_{i \neq \wi} \| u'_i \|^2_{L^2 (\mathcal H_{i})}
 +  \int_0^{\widetilde x} | u'_{\wi} |^2 dx + 
 \int_{\widetilde x}^{+\infty} | u'_{\wi} |^2 dx
 -  \int_0^{\widetilde x} | v'_{\wi} |^2 dx - 
 \int_{\widetilde x}^{+\infty} | v'_{\wi} |^2 dx \\
= &  
 \sum_{i \neq \wi} \| u'_i \|^2_{L^2 (\mathcal H_{i})}
 +  
 \int_{\widetilde x}^{+\infty} | u'_{\wi} |^2 dx
 -  
 \int_{\widetilde x}^{+\infty} | v'_{\wi} |^2 dx
\\
= &  
 \sum_{i \neq \wi} \| u'_i \|^2_{L^2 (\mathcal H_{i})}
 +  
 \int_{\widetilde x}^{+\infty} | u'_{\wi} |^2 dx
 -  
 \int_{\widetilde x}^{+\infty} | (u^*)' (x - \widetilde x) |^2 dx
 \\
 \geq & \ 0,
  \end{split}
\end{equation}  
where in the last passage we used inequality Eq \eqref{polyaszego}.

Then, since $u$ is supposed to be a Ground State for $E (\cdot, \G)$, by Eq \eqref{final} it must be 
$$E (v, \G') \leq E (u, \G) \leq - \theta_p \mu^{2 \beta + 1},$$
therefore by the existence criterion there exists a Ground State for
$E (\cdot, \G')$ at mass $\mu,$ that contradicts the present proof in the case $n = 1.$
This concludes the proof.
\end{proof}

\begin{remark}{\em 
As an application of Theorem \ref{th:nonex} we consider the graph
$\mathcal{G}$
made of one halfline and a compact core $\mathcal{K}$ consisting of $n$ edges $e_i$, $i=1, \dots n,$ each of length $l$, all attached at the origin of the halfline (see Fig. \ref{nfork}).
Thus  ${\rm{diam}}({\mathcal{K}}) = 2 l$ and $|{\mathcal{K}}| = nl$.

Moreover, we take in consideration a potential $w$ supported on the edges $e_i$ and defined
as 
$$w_i (x) : = \frac{\epsilon^3}{4l^{2k +2}} x^{2 k}, \quad i =1, \dots ,n, \ x \in [0,l], \ k \in \mathbb{N},$$
where $w_i$ denotes the restriction of $w$ on the edge $e_i.$
Obviously $w \geq 0$ and $\| w \|_\infty = \frac{\epsilon^3}{4 l^2}$. 

The condition \eqref{theor} rewrites then as
$$ \epsilon > \max \left( 2 \mu^\beta l, \, \frac 1 {\mu^\beta n l}, \,
\frac{\epsilon^3}{4 l^2\mu^{2 \beta}} \right),$$
that, by a straightforward computation, amounts to
\begin{equation}
\label{muchoice}
     \frac{1}{nl\epsilon} < \mu^\beta < \frac{\epsilon}{2 l}.
     \end{equation}
     
The inequalities Eq \eqref{muchoice} can be simultaneously satisfied for $n$ large enough. In other words, if $n$ is large enough, then
there exists an interval of masses to which Theorem \ref{th:nonex} applies.

This example shows that Theorem \ref{th:nonex} is not empty, in the sense that for some graphs the condition Eq \eqref{theor} singles out a significant interval of masses.

At present, we do not have a precise estimate of  the constant $\epsilon.$ It will be the subject of further investigation.
     
}
\end{remark}

\begin{figure}[h!]
 \centering
 \includegraphics[height = 4 cm ]{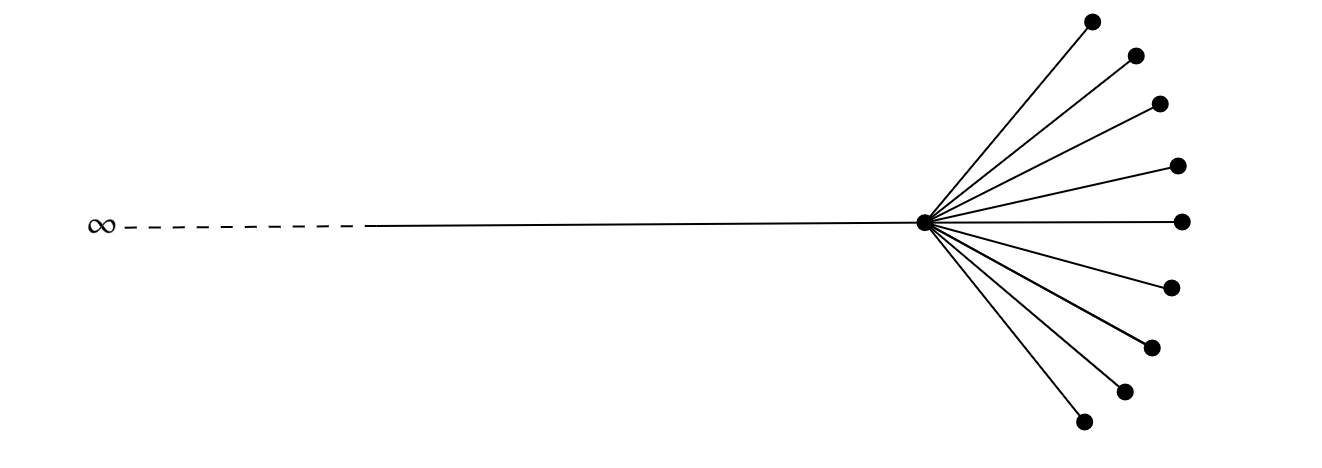}
        \caption{A n-fork graph consisting of one halfline and $n$ edges of length $l$.}\label{nfork}
\end{figure}

\section*{Use of AI tools declaration}
The authors declare they have not used Artificial Intelligence (AI) tools in the creation of this article.

\section*{Acknowledgments}

R. Adami acknowledges that this study was carried out within the project E53D23005450006 ``Nonlinear dispersive equations in presence of singularities” funded by European Union-Next Generation EU within the PRIN 2022 program (D.D. 10402/02/2022 Ministero dell’Università e della Ricerca). This manuscript reflects only the authors’ views and opinions and the Ministry cannot be considered responsible for them. %\\
R. Adami has been partially supported by the INdAM Gnampa 2023 project ``Modelli nonlineari in presenza di interazioni puntuali”. %\\
D. Spitzkopf has been partially supported by the European Union's Horizon 2020 research and innovation programme under the Marie Sk$\l$odowska-Curie grant agreement No 873071. 

\section*{Conflict of interest}

The authors declare there is no conflicts of interest.

%\section*{ Authors contribution }

%All authors participate at the same level in ideas and writing.

\end{document}